\documentclass[11pt]{amsart}
\usepackage{amssymb}
\usepackage{amsmath}
\usepackage{amsfonts}
\usepackage{graphicx}

\usepackage{hyperref}
    \usepackage{aeguill}
    \usepackage{type1cm}

\theoremstyle{plain}
\newtheorem{thm}{Theorem}[section]

\newtheorem{claim}[thm]{Claim}

\newtheorem{corollary}[thm]{Corollary}

\newtheorem{lemma}[thm]{Lemma}

\newtheorem{remark}[thm]{Remark}

\newtheorem{theorem}[thm]{Theorem}

\newtheorem{ex}[thm]{Example}

\numberwithin{equation}{section}
\newcommand{\N}{\mathbb{N}}

\newcommand{\R}{\mathbb{R}}

\usepackage[usenames,dvipsnames]{color}

\usepackage[dvipsnames]{xcolor}

\begin{document}

\title[The Morse-Sard theorem revisited]{The Morse-Sard theorem revisited}

\author{D. Azagra}
\address{ICMAT (CSIC-UAM-UC3-UCM), Departamento de An{\'a}lisis Matem{\'a}tico y Matem\' atica Aplicada,
Facultad Ciencias Matem{\'a}ticas, Universidad Complutense, 28040, Madrid, Spain}
\email{azagra@mat.ucm.es}

\author{J. Ferrera}
\address{IMI, Departamento de An{\'a}lisis Matem{\'a}tico y Matem\' atica Aplicada,
Facultad Ciencias Matem{\'a}ticas, Universidad Complutense, 28040, Madrid, Spain}
\email{ferrera@mat.ucm.es}

\author{J. G\'omez-Gil}
\address{Departamento de An{\'a}lisis Matem{\'a}tico y Matem\' atica Aplicada,
Facultad Ciencias Matem{\'a}ticas, Universidad Complutense, 28040, Madrid, Spain}
\email{gomezgil@mat.ucm.es}

\date{December 18, 2015}


\thanks{The authors were partially supported by Grant MTM2012-34341. D. Azagra was partially supported by Ministerio de Educaci\'on, Cultura y Deporte, Programa Estatal de Promoci\'on del Talento y su Empleabilidad en I+D+i, Subprograma Estatal de Movilidad.}


\begin{abstract}
Let $n, m, k$ be positive integers with $k=n-m+1$. We establish an abstract Morse-Sard-type theorem which allows us to deduce, on the one hand, a previous result of De Pascale's for Sobolev $W^{k,p}_{\textrm{loc}}(\mathbb{R}^n, \mathbb{R}^m)$ functions with $p>n$ and, on the other hand, also some new results such as the following: if $f\in C^{k-1}(\mathbb{R}^n, \mathbb{R}^m)$ satisfies $$\limsup_{h\to 0}\frac{|D^{k-1}f(x+h)-D^{k-1}f(x)|}{|h|}<\infty$$ for every $x\in\mathbb{R}^n$ (that is, $D^{k-1}f$ is a Stepanov function), then the set of critical values of $f$ is Lebesgue-null in $\mathbb{R}^m$. In the case that $m=1$ we also show that this limiting condition holding for every $x\in\R^n\setminus\mathcal{N}$, where $\mathcal{N}$ is a set of zero $(n-2+\alpha)$-dimensional Hausdorff measure for some $0<\alpha<1$, is sufficient to guarantee the same conclusion. 
\end{abstract}

\maketitle

\section{Introduction and main results}

The Morse-Sard theorem \cite{Morse, Sard} states that if $f:\R^n\to\R^m$ is of class $C^k$, where $k=n-m+1$, then the set of critical values of $f$ has measure zero in $\R^m$. A famous example of Whitney's \cite{Whitney} shows that this classical result is sharp within the classes of functions $C^j$. However, several generalizations of the Morse-Sard theorem for other classes of functions (notably H\"older and Sobolev spaces) have appeared in the literature; see 
\cite{Alberti, Bates, BoHaSt, BourKoKris1, BourKoKris2, DePascale, Dubo, Figalli, KorobkovKristensen, HajlaszZimmerman, Norton, NortonZMS, Landis, PavZaj, Putten, Rifford, Yomdin} and the references therein.

In this paper we will give a new, elementary proof of the Morse-Sard theorem which only requires $f$ to be $k-1$ times continuously differentiable, to have a Taylor expansion of order $k$ at almost every point $x\in\R^n$, and to satisfy $\limsup_{h\to 0}\frac{|f(x+h)-f(x)-Df(x)(h) - ... - \frac{1}{(k-1)!} D^{k-1}f(x)(h^{k-1})|}{|h|^k}<\infty$ for every $x\in\R^n$. We will also see that this new version of the Morse-Sard theorem is neither weaker nor stronger than the recent results of Bourgain, Korobkov and Kristensen \cite{BourKoKris2} for $W^{n,1}_{\textrm{loc}}$ and $BV_{n,\textrm{loc}}$ functions in the case $m=1$, and it can be useful in the analysis of functions $f$ of class $C^{k-1}(\R^n, \R^m)$ whose derivatives $D^{k-1}f$ behave somewhat badly near some points or regions, but nonetheless $f$ has Taylor expansions almost everywhere.
We will deduce our result of a new, abstract Morse-Sard-type theorem which will also allow us to easily recover De Pascale's theorem \cite{DePascale} for $W^{k,p}(\R^n, \R^ m)$ Sobolev functions with $p>n$. In any case, it should be observed that the results of this paper, as those of \cite{Bates, DePascale, HajlaszZimmerman} (and leaving aside the special case $n=m$), concern everywhere differentiable functions, while the results of \cite{BourKoKris1, BourKoKris2, KorobkovKristensen} are valid for Sobolev functions that need not be everywhere differentiable. Thus one can see that in this kind of problems there are mainly two issues, namely the Luzin property and the Morse-Sard property, and that while \cite{BourKoKris1, BourKoKris2, KorobkovKristensen} deal with both issues simultaneously, in this paper we will only be concerned with the second.

We will say that a function $f:\R^n\to\R^m$ has a Taylor expansion of order $k$ at a point $x$ provided there is a polynomial $P_x:\R^n\to\R^m$, of degree less than or equal to $k$, such that 
\begin{equation}\label{definition of Taylor expansion}
\lim_{h\to 0}\frac{f(x+h)-P_x(h)}{|h|^k}=0.
\end{equation}
Whenever such a $P_x$ exists, it is unique, and we will denote $\{P_x\}=J^{k}f(x)$ (read the $k$-th order jet of $f$ at $x$). Abusing notation, we will write indistinctly
$
J^{k}f(x)= (P^{1}_{x}, ..., P^{k}_{x})
$
and also
$$
(P^{1}_{x}, ..., P^{k}_{x}) \in J^{k}f(x),
$$
where $P_{x}(h)=f(x)+P^{1}_{x}(h)+ ... +P^{k}_{x}(h)$, and $P_{x}^{j}$ is the $j$-homogeneous polynomial component of $P_x$, for each $j=1, 2, ..., k$.  If there exists no such $P_x$ then we will write $J^{k}f(x)=\emptyset$.

If $f$ happens to be $k$ times differentiable at $x$ then it is known that $f$ has a Taylor expansion $P_x$ of order $k$ at $x$, and in fact we have $P^{j}_{x}(h)=\frac{1}{j!}D^{j}f(x)(h^{j})$ for every $j=1, ..., k$. However, a function $f$ may have a Taylor expansion of order $k$ at $x$ without being $k$ times differentiable at $x$ (or even two times differentiable at $x$, no matter how large $k$ is). 

The set of all functions $f:\R^n\to\R^m$ of class $C^{j}$ such that  $f$ has a Taylor expansion of order $k$ at each point $x\in\R^n$ is obviously a vector space, which we will denote by $\mathcal{C}^{j}\mathcal{P}^{k}(\R^n, \R^m)$ in this paper.

It should be noted that our definition of function admitting Taylor expansions on a set is much less demanding than other definitions appearing in the literature (compare e.g. with \cite[Chapter 3.5]{Ziemer}). Specifically, we do not require that the mappings $x\mapsto P_x$ be locally bounded, nor that the limits in \eqref{definition of Taylor expansion} be locally uniform in $x$. Consequently, functions admitting Taylor expansions of order $k$ at all points of a compact set are not, in general, restrictions of $C^k$ functions.

From one of our main results we will deduce the following.

\begin{theorem}\label{general theorem}
Let $n\geq m$ be positive integers, $k:=n-m+1$ and let $f:\R^n\to\R^m$ be such that
\begin{enumerate}
\item $f\in C^{k-1}(\R^n, \R^m)$;
\item $\limsup_{h\to 0}\frac{|f(x+h)-f(x)-Df(x)(h) - ... - \frac{1}{(k-1)!} D^{k-1}f(x)(h^{k-1})|}{|h|^k}<\infty$ for every $x\in\R^n$.
\end{enumerate}
Then $\mathcal{L}^{m}\left(f(C_f) \right)=0$, where
$C_f:=\{ x\in\R^n : \textrm{rank}\left(Df(x)\right) < m\}$.

\noindent The same statement holds true if $\R^n$ is replaced with an open subset of $\R^n$. 
\end{theorem}
\noindent Here, as in the rest of the paper, $\mathcal{L}^{m}$ denotes the Lebesgue outer measure in $\R^m$.
Notice that for $n>m$ the critical set $C_{f}=\{ x\in\R^n : \textrm{rank}\left(Df(x)\right) \leq m-1\}$ is defined as usual, as $f$ must be at least of class $C^1$. In the case $n=m$, we note that condition $(2)$ reads
$\limsup_{h\to 0}\frac{|f(x+h)-f(x)|}{|h|}<\infty$ for every $x\in\R^n$, which by itself already implies, by Stepanov's differentiability theorem, that $f$ is differentiable almost everywhere, and therefore we may define the critical set of $f$ by $C_{f}=\{ x\in\R^n : Df(x) \textrm{ exists and } \textrm{rank}\left(Df(x)\right) \leq n-1\}$. In this case the above Theorem tells us that $\mathcal{L}^n(f(C_f))=0$. 

After the first version of this paper was released, by using more advanced methods we have been able to improve Theorem \ref{general theorem}, especially in the case $m=1$; see\cite[Theorem 1.6]{AFGGNonlinear2017} and the Appendix therein.

A straightforward consequence of Theorem \ref{general theorem} is the following.
\begin{theorem}\label{Stepanov meets Morse-Sard}
Assume that $f\in C^{k-1}(\R^n, \R^m)$, with $k=n-m+1\geq 1$, and that $D^{k-1}f$ satisfies
$$
\limsup_{h\to 0}\frac{\|D^{k-1}f(x+h)-D^{k-1}f(x)\|}{|h|}<\infty
$$
for every $x\in\R^n$. Then $\mathcal{L}^{m}\left( f(C_f)\right)=0$.

\noindent The same statement holds true if $\R^n$ is replaced with an open subset of $\R^n$. 
\end{theorem}
It is clear that the above Theorem generalizes the main result of \cite{Bates}.  On the other hand, observe that if $J^k{f}(x)=(P_{x}^1, ..., P_{x}^k)$ and $f\in C^{k-1}(\R^n, \R^m)$ then $$\limsup_{h\to 0}\frac{|f(x+h)-f(x)-Df(x)(h) - ... - \frac{1}{(k-1)!} D^{k-1}f(x)(h^{k-1})|}{|h|^k}=\|P_{x}^{k}\|<\infty.$$
Thus, another consequence of Theorem \ref{general theorem} is the following.

\begin{corollary}\label{main theorem}
Let $n, m$ be positive integers with $n\geq m$, and set $k=n-m+1$. If $f\in \mathcal{C}^{k-1}\mathcal{P}^{k}(\R^n, \R^m)$ then the set of critical values of $f$ is of Lebesgue measure zero in $\R^m$. 

\noindent The same statement holds true if $\R^n$ is replaced with an open subset of $\R^n$.
\end{corollary}
Obviously, this Corollary also implies the classical version of the Morse-Sard theorem for functions $f:\R^n\to\R^m$ of class $C^k$. 
It is natural to ask whether condition $(2)$ of Theorem \ref{general theorem} could be replaced with the same limit condition holding for a.e. $x$ instead of all $x$. The answer is negative, as can be ascertained by examining Whitney's classical example \cite{Whitney}. See nonetheless Theorems \ref{general theorem variant 1} and \ref{general theorem variant 2} below for some refinements of Theorem \ref{general theorem} in the case $m=1$, regarding this question and Norton's results from \cite{Norton}. In particular we obtain the following.
\begin{corollary}\label{corollary by Norton}
If a function $f\in C^{n-1}(\mathbb{R}^n, \mathbb{R})$ satisfies $$\limsup_{h\to 0}\frac{|D^{n-1}f(x+h)-D^{n-1}f(x)|}{|h|}<\infty$$ for every $x\in\R^n\setminus\mathcal{N}$, where $\mathcal{N}$ is a set such that $\mathcal{H}^{n-2+\alpha}(\mathcal{N})=0$ for some $0<\alpha<1$, then
$\mathcal{L}^{1}(f(C_{f}))=0$.
\end{corollary}

We will deduce Theorem \ref{general theorem} from the following abstract, more powerful Morse-Sard-type result.

\begin{theorem}\label{abstract theorem}
For all positive integers $n\geq m\geq 1$, $2\leq j\leq n$, and open sets $U$ and $V$ in $\R^n$ and $\R^m$ respectively, let $\mathcal{C}_{n,m}^{j}(U,V)$ be classes of mappings $f:U\to V$ such that:
\begin{enumerate}
\item[{(MS1)}] $C^{j}(U, V) \subset\mathcal{C}_{n,m}^{j}(U,V)\subset C^{j-1}(U, V)$;
\item [{(MS2)}] if $n\geq m+1$, $f\in\mathcal{C}_{n,m}^{n-m+1}(U,V)$, $A$ is a subset of $\{x: D^{i}f(x)=0 \textrm{ for } i=1, ..., n-m\}$ and $\mathcal{L}^n(A)=0$, then $\mathcal{L}^{m}(f(A))=0$;
\item[{(MS3)}] if $n\geq m+1$ then every $f\in\mathcal{C}_{n,m}^{n-m+1}(U,V)$ has a Taylor expansion of order $k=n-m+1$ at almost every $x\in\R^n$;
\item[{(MS4)}] if $f\in \mathcal{C}_{n,n}^{j}(U,V)$ and $f$ maps $U$ diffeomorphically onto $V$ then $f^{-1}\in \mathcal{C}_{n,n}^{j}(V,U)$;
\item[{(MS5)}] if $f\in \mathcal{C}_{n,m}^{j}(U,V)$, $g\in \mathcal{C}_{n,n}^{j}(W,U)$, and either $f$ is of class $C^{j}$ or $g$ is a diffeomorphism, then $f\circ g\in \mathcal{C}_{n,m}^{j}(W,V)$;
\item[{(MS6)}] if $g\in \mathcal{C}_{n,m}^{j}(U,V)$, $1\leq i\leq m-1$, then $g_y\in \mathcal{C}_{n-i,m}^{j}(U_y, \R^m)$ for $\mathcal{L}^{i}$-a.e. $y\in\R^{i}$ with $U_y:=\{z\in\R^{n-i} : (y,z)\in U\}\neq\emptyset$, where $g_{y}:U_y\to\R^m$ is defined by $g_{y}(z)=g(y,z)$;  
\item[{(MS7)}] if $f\in\mathcal{C}_{n,i}^{j}(U,V)$ and $g\in\mathcal{C}_{n,s}^{j}(U,V')$ with $i+s=m$, then $\varphi\in\mathcal{C}_{n,m}^{j}(U,V\times V')$, where $\varphi$ is defined by $\varphi(x)=(f(x),g(x))$.
\end{enumerate}
Assume that $n\geq m+1$. Then $\mathcal{L}^{m}(f(C_f))=0$ for every $f\in \mathcal{C}_{n,m}^{n-m+1}(U,V)$.

Moreover, in the special case $m=1$, the classes $C^{j}_{n,1}(U,V)$ may be defined only for $j=n$, and the conditions (MS1) with $j=n$, (MS2) and (MS3) alone are sufficient to ensure that $\mathcal{L}^{1}(f(C_f))=0$ for every $f\in \mathcal{C}_{n,1}^{n}(U,V)$. 
\end{theorem}

Another consequence of the above result is the following theorem of De Pascale \cite{DePascale} for Sobolev spaces (see also \cite{Figalli} for a simpler proof).

\begin{theorem}[De Pascale, 2001]\label{De Pascale}
Let $n, m, k$ be positive integers with $n\geq m$, $k=n-m+1$, and let $p$ be a real number with $p>n$. Then
$\mathcal{L}^{m}(f(C_f))=0$
for every $f$ in the Sobolev space $W^{k,p}_{\textrm{loc}}(\R^n, \R^m)$. 
\end{theorem}

We will also establish the following Dubovitski\v{\i}-Sard-type versions of Theorems \ref{abstract theorem} and \ref{general theorem}.
 
\begin{theorem}\label{abstract D-S theorem}

For all positive integers $n\geq m\geq 1$, $2\leq j\leq n$, and open sets $U$ and $V$ in $\R^n$ and $\R^m$ respectively, let $\mathcal{C}_{n,m}^{j}(U,V)$ be classes of mappings $f:U\to V$ such that:
\begin{enumerate}
\item[{(DS1)}] $C^{j}(U, V) \subset\mathcal{C}_{n,m}^{j}(U,V)\subset C^{j-1}(U, V)$;
\item [{(DS2)}] if $n\geq m+1$, $2\leq k\leq n-m+1$, $\ell=n-m-k+1$, $f\in\mathcal{C}_{n,m}^{k}(U,V)$, $A$ is a subset of $\{x: D^{j}f(x)=0 \textrm{ for } j=1, ..., k-1\}$ and $\mathcal{L}^n(A)=0$, then $\mathcal{H}^{\ell}(A\cap f^{-1}(y))=0$ for a.e. $y\in\R^m$;
\item[{(DS3)}] if $n\geq m+1$ and $2\leq k\leq n-m+1$ then every $f\in\mathcal{C}_{n,m}^{k}(U,V)$ has a Taylor expansion of order $k$ at almost every $x\in\R^n$;
\item[{(DS4)}] if $f\in \mathcal{C}_{n,n}^{j}(U,V)$ and $f$ maps $U$ diffeomorphically onto $V$ then $f^{-1}\in \mathcal{C}_{n,n}^{j}(V,U)$;
\item[{(DS5)}] if $f\in \mathcal{C}_{n,m}^{j}(U,V)$, $g\in \mathcal{C}_{n,n}^{j}(W,U)$, and either $f$ is of class $C^{j}$ or $g$ is a diffeomorphism, then $f\circ g\in \mathcal{C}_{n,m}^{j}(W,V)$;
\item[{(DS6)}] if $g\in \mathcal{C}_{n,m}^{j}(U,V)$, $1\leq i\leq m-1$, then $g_y\in \mathcal{C}_{n-i,m}^{j}(U_y, \R^m)$ for $\mathcal{L}^{i}$-a.e. $y\in\R^{i}$ with $U_y:=\{z\in\R^{n-i} : (y,z)\in U\}\neq\emptyset$, where $g_{y}:U_y\to\R^m$ is defined by $g_{y}(z)=g(y,z)$;  
\item[{(DS7)}] if $f\in\mathcal{C}_{n,i}^{j}(U,V)$ and $g\in\mathcal{C}_{n,s}^{j}(U,V')$ with $i+s=m$, then $\varphi\in\mathcal{C}_{n,m}^{j}(U,V\times V')$, where $\varphi$ is defined by $\varphi(x)=(f(x),g(x))$.
\end{enumerate}
Assume that $n\geq m+1$, $2\leq k\leq n-m+1$, and $\ell=n-m-k+1$. Then, for every $f\in \mathcal{C}_{n,m}^{k}(U,V)$ we have
$$
\mathcal{H}^{\ell}(C_{f}\cap f^{-1}(y))=0 \textrm{ for a.e. } y\in\R^m.
$$
Moreover, in the special case $m=1$, the conditions (DS1)-(DS3) alone are sufficient to ensure that, if $2\leq k \leq n-m+1$ and $\ell=n-m-k+1$, then for every $f\in \mathcal{C}_{n,1}^{n}(U,V)$ one has $\mathcal{H}^{\ell}(C_{f}\cap f^{-1}(y))=0$ for a.e. $y\in\R$.
\end{theorem}
\noindent Here, as in the rest of the paper, $\mathcal{H}^{\ell}$ denotes the $\ell$-dimensional outer Hausdorff measure in $\R^n$. 
\begin{theorem}\label{general D-S theorem}
Let $n\geq m$ be positive integers, $1\leq k\leq n-m+1$, and let $f:\R^n\to\R^m$ be such that
\begin{enumerate}
\item $f\in C^{k-1}(\R^n, \R^m)$;
\item $\limsup_{h\to 0}\frac{|f(x+h)-f(x)-Df(x)(h) - ... - \frac{1}{(k-1)!} D^{k-1}f(x)(h^{k-1})|}{|h|^k}<\infty$ for every $x\in\R^n$.
\end{enumerate}
Then, for $\ell=n-m-k+1$, we have that
$\mathcal{H}^{\ell}(C_{f}\cap f^{-1}(y))=0$ for a.e. $y\in\R^m$.
\end{theorem}

It is clear that Theorems \ref{abstract D-S theorem} and \ref{general D-S theorem} include Theorems \ref{abstract theorem} and \ref{general theorem} as particular cases for $k=n-m+1$, or equivalently $\ell=0$. We chose to state them separately for expository reasons.

Another consequence of Theorem \ref{abstract D-S theorem} is the following Theorem of P. Haj\l asz and S. Zimmerman \cite{HajlaszZimmerman}

\begin{theorem}[Haj\l asz-Zimmerman]\label{HajlaszZimmerman}
Fix $n, m, k\in\N$. Suppose $\Omega\subseteq\R^n$ is open and $f\in W^{k,p}_{\textrm{loc}}(\Omega, \R^m)$ for some $n<p<\infty$. If $\ell=\max\{n-m-k+1, 0\}$ then
$\mathcal{H}^{\ell}(C_{f}\cap f^{-1}(y))=0$ for a.e. $y\in\R^m$.
\end{theorem}

Both Theorems \ref{general D-S theorem} and \ref{HajlaszZimmerman} are generalizations of the so-called Dubovitski\v{\i}-Sard theorem \cite{Dubo}, which guarantees the same conclusion under the more stringent assumption that $f\in C^{k}(\R^n, \R^m)$.

Of course the assumptions of Theorems \ref{abstract theorem} and \ref{abstract D-S theorem}, as well as the introduction of the classes $\mathcal{C}^{j}_{n,m}$, may look rather artificial (and, as the reader may suspect, they are tailored to the proof we give). Nonetheless we believe they are quite useful, as they allow us to provide simple and elementary proofs of known results (cf. \cite{Bates} and Theorems \ref{De Pascale} and \ref{HajlaszZimmerman}; see also Remark \ref{simplicity}), and to simultaneously obtain new results such as Theorem \ref{Stepanov meets Morse-Sard} and Corollary \ref{corollary by Norton}.

The rest of this paper is organized as follows. In Section 2 we will collect several known or easy results that will facilitate the proof of Theorem \ref{abstract theorem}, which we will provide in Section 3.
In section 4 we will explain how Theorem \ref{general theorem} is deduced from Theorem \ref{abstract theorem}. In Section 5 we will explain how Theorem \ref{De Pascale} can be deduced from Theorem \ref{abstract theorem} with little effort, we will study some examples that will clarify the relations of the above results to other authors' work, and we will state and prove some variants of Theorem \ref{general theorem} for $m=1$. Finally, in Section 6 we will explain what ingredients have to be added to the proofs of  Theorems \ref{abstract theorem} and \ref{general theorem} in order to obtain Theorems  \ref{abstract D-S theorem} and \ref{general D-S theorem} (again, we choose to separate proofs for expository reasons and because the Dubovitski\v{\i}-Sard-type results require using more advanced tools, such as the upper integral and Hausdorff measures).

\medskip

\section{Some tools}

In the proof of our main results it will be very convenient to use the Kneser-Glaeser Rough Composition Theorem, which we next restate (see \cite[Theorem 14.1]{AbrahamRobbin} for its proof, based on an application of Whitney's Extension Theorem \cite{Whitney1934}). Let us recall that, given a positive integer $s$, a map $f$ is said to be $s$-flat on a set $A$ if $D^{j}f(x)=0$ for every $x\in A$ and $j=1, ..., s$.

\begin{theorem}[Kneser-Glaeser]\label{Kneser Glaeser}
Let $W\subset\R^m$ and $V\subset\R^n$ be open sets; $A^{*}\subset W$ and
$A\subset V$, with $A$ closed relative to $V$, $f:V\to\R^p$ of class $C^r$ on $V$ and $s$-flat on $A$, $g:W\to V$ of class $C^{r-s}$ with $g(A^{*})\subset A$. Then there is a map $H:W\to\R^{p}$ of class $C^r$ satisfying:
\begin{enumerate}
\item $H(x)=f(g(x))$ for $x\in A^{*}$;
\item $H$ is $s$-flat on $A^{*}$.
\end{enumerate}
\end{theorem}

We next state and prove eight lemmas that we will also need in the next two sections.

\begin{lemma}\label{n null sets are mapped onto m null sets}
Let $n\geq m$, $k=n-m+1$, and $f:\R^n\to\R^m$ be a function. Assume that $N$ is a subset of $\{x: \limsup_{h\to 0}\frac{|f(x+h)-f(x)|}{|h|^k}<\infty\}$, and that $\mathcal{L}^n(N)=0$. Then $\mathcal{L}^m(f(N))=0$ as well.
\end{lemma}
\begin{proof}
For each $j\in\N$ we define $$A_j=\{x\in\R^n :
\limsup_{h\to 0}\frac{|f(x+h)-f(x)|}{|h|^k}\leq j\}.$$
It is clear that $N\subset\bigcup_{j=1}^{\infty} A_j$, hence it is enough to see that $\mathcal{L}^m(f(N\cap A_j))=0$ for each $j$.

Let us fix $j\in\N$, $\varepsilon>0$. Since $\mathcal{L}^n(N)=0$, for each $i\in\N$ there exists a sequence $\{C_{i \alpha}\}_{\alpha}$ of cubes with $\textrm{diam}(C_{i \alpha})\leq 1/i$, $N\subset \bigcup_{\alpha}C_{i \alpha}$, and
$$\sum_{\alpha}\textrm{vol}(C_{i \alpha})<\frac{\varepsilon}{n^{n/2}\alpha(m) (j+1)^m},$$
where $\alpha(m)$ denotes the volume of the unit ball of $\R^m$.
Define now
$$
D_i=\{x\in A_j\cap N : |f(x+h)- f(x)|\leq (j+1)|h|^k \textrm{ whenever } |h|<\frac{1}{i}\}.
$$
It is easily checked that $D_i\subseteq D_{i+1}$ for every $i$, and
$A_j\cap N= \bigcup_{i=1}^{\infty}D_i$.

If $x, y\in D_i\cap C_{i \alpha}$ we have $|f(x)-f(y)|\leq (j+1)|x-y|^k$ and therefore, using that $km=(k-1)(m-1)+n \geq n$, we get
$$
|f(x)-f(y)|^m\leq (j+1)^m \textrm{diam}(C_{i \alpha})^{km}\leq
(j+1)^m \textrm{diam}(C_{i \alpha})^{n},
$$
which implies that
$$
\textrm{diam}(f(D_i\cap C_{i \alpha}))^{m}\leq
(j+1)^m \textrm{diam}(C_{i \alpha})^{n},
$$ 
and consequently
$$
\mathcal{L}^m(f(D_i\cap C_{i \alpha}))\leq (j+1)^m n^{n/2} \alpha(m)\textrm{vol}(C_{i \alpha})
$$
as well.
By \cite[Theorem 1.1.2, p.5]{EvansGariepy} we obtain
\begin{eqnarray*}
& &\mathcal{L}^m(f(A_j\cap N))=\lim_{i\to\infty}\mathcal{L}^m(f(D_i))\leq
\limsup_{i\to\infty}\sum_{\alpha}\mathcal{L}^m (f(D_i\cap C_{i \alpha}))\leq\\
& &\limsup_{i\to\infty}\sum_{\alpha}(j+1)^m n^{n/2} \alpha(m)\textrm{vol}(C_{i \alpha}) \leq \varepsilon,
\end{eqnarray*}
and by letting $\varepsilon$ go to $0$ we conclude that $\mathcal{L}^m(f(N\cap A_j))=0$.
\end{proof}

\begin{lemma}\label{points with 0 jets are good}
Let $f:\R^n\to\R^m$ be a function, $n\geq m$, $k=n-m+1$. Then
$$
\mathcal{L}^m\left(f\left(\{x\in\R^n : (0, ..., 0)\in J^{k}f(x)\}\right)\right)=0.
$$
\end{lemma}
\begin{proof}
Let us denote $A=\{x\in\R^n : (0, ..., 0)\in J^{k}f(x)\}$. We want to see that $\mathcal{L}^m(f(A))=0$. By using obvious truncation arguments we may assume, without loss of generality, that $A\subset B:=[-R,R]^{n}$ for some fixed $R\in\N$. Let $\varepsilon >0$. For every $i\in \N$ we define
$$
D_i=\{ x\in B : |f(x+h)-f(x)|\leq \varepsilon |h|^k \textrm{ whenever } |h|\leq \frac{\sqrt{n}}{i}\}.
$$
Note that $\{ D_i\}$ is an increasing sequence of sets such that $A\subset \bigcup _j D_j $ (indeed, for every $x\in A$ we have
$$
\lim_{h\to 0}\frac{f(x+h)-f(x)}{|h|^k}= 0,
$$
and consequently $|f(x+h)-f(x)|\leq \varepsilon |h|^k$ if $|h|\leq \frac{\sqrt{n}}{i}$ for $i$ big enough).
For each $i$, we may decompose $B$ as the union of a family of cubes $\{C_{i \alpha}\}_{\alpha=1}^{N(i)}$ of diameter  $\sqrt{n}/i$ (hence of volume $1/i^n$) with pairwise disjoint interiors. In particular 
$\sum_{\alpha}\textrm{vol}(C_{i \alpha})=\sum_{\alpha}\frac{1}{i^n}=\textrm{vol}(B)$.

If $x,y\in D_i\cap C_{i \alpha}$ then $|x-y|\leq \frac{\sqrt{n}}{i}$, hence
\begin{equation}\label{function well bounded}
|f(y)-f(x)|\leq \varepsilon |x-y|^k\leq \frac{\varepsilon n^{k/2}}{i^k},
\end{equation}
and (recalling that $km\geq n$)
\begin{equation}\label{diameters well bounded}
|f(y)-f(x)|^m\leq  \frac{\varepsilon^m n^{km/2}}{i^{km}}\leq 
\frac{\varepsilon^m n^{km/2}}{i^{n}}=\varepsilon^m \, n^{km/2}\textrm{vol}(C_{i \alpha}).
\end{equation}
It follows that
$$
\mathcal{L}^m\left(f(D_i\cap C_{i \alpha})\right)\leq \varepsilon^m\alpha(m)n^{km/2}\textrm{vol}(C_{i \alpha}),
$$
and therefore
\begin{eqnarray*}
& &\mathcal{L}^m(f(D_i))\leq \sum_{\alpha} \mathcal{L}^m (f(D_i\cap C_{i \alpha}))\\
& & \leq\varepsilon^m\alpha(m)n^{km/2}\sum_{\alpha}\textrm{vol} (C_{i \alpha})=
\varepsilon^m\alpha(m)n^{km/2}\textrm{vol}(B).
\end{eqnarray*}
Using \cite[Theorem 1.1.2, p.5]{EvansGariepy} we obtain
\begin{eqnarray*}
& &\mathcal{L}^m(f(A))\leq \lim_{i\to\infty}\mathcal{L}^m(f(D_i))\leq \varepsilon^{m} \alpha(m) n^{km/2}\textrm{vol}(B),
\end{eqnarray*}
and by letting $\varepsilon$ go to $0$ we conclude that  $\mathcal{L}^m(f(A))=0$.
\end{proof}

\begin{lemma}\label{Cantor}
Let $C$ be a bounded subset of $\R$ with the following property:
there exists $\alpha \in (0, \frac{1}{2})$ such that, for every $x,y\in C$, the interval
$$I_{xy}=\bigl( \frac{x+y}{2}-\alpha |y-x|,\frac{x+y}{2}+\alpha |y-x|\bigr)
$$
does not intersect $C$. Then $\mathcal{L}^{1}(C)=0$.
\end{lemma}
\begin{proof}
Take $r$ so that $2\alpha< r<1$. By the hypothesis, for every  interval
$I$, $\mathcal{L}^{1}(C\cap I)\leq (1-r)\mathcal{L}^{1}(I)$. As a
consequence, if  $(I_k)_k
$ is a sequence of intervals such that $C\subset \cup_k I_k$ then
\begin{equation*}
\mathcal{L}^{1}(C)\leq \sum_k \mathcal{L}^{1}(C\cap I_k)\leq
(1-r)\sum_k \mathcal{L}^{1}(I_k)
\end{equation*}
which implies that $\mathcal{L}^{1}(C)\leq (1-r)\mathcal{L}^{1}(C)$
and therefore  that $\mathcal{L}^{1}(C)=0$.
\end{proof}

\begin{lemma}\label{inverses are OK1}
Let $k\geq 2$ be a positive integer and $\varphi:U\to V$ be a $C^{k-1}$ diffeomorphism between two open subsets of $\R^n$. Assume that $\varphi$ has a Taylor expansion of order $k$ at $x$. Then the inverse diffeomorphism $\varphi^{-1}$ has a Taylor expansion of order $k$ at $y=\varphi(x)$.

In particular, if $\varphi$ has a Taylor expansion of order $k$ almost everywhere then so does $\varphi^{-1}$.
\end{lemma}
\begin{proof}
Without loss of generality we may assume that $x=0$ and $y=\varphi(0)=0$. Write $\varphi(h)=P(h)+R(h)$, where $R(h)=o(|h|^k)$ and $P$ is a polynomial of degree less than or equal to $k$. We have
$$
h=\varphi(\varphi^{-1}(h))= P(\varphi^{-1}(h))+R(\varphi^{-1}(h)).
$$
Since $DP(0)=D\varphi(0)$ is invertible, there is an open neighborhood of $0$, which we may assume to be $U$, on which $P$ is a $C^{k-1}$ diffeomorphism, with an inverse denoted by $P^{-1}$. It follows that
$$
\varphi^{-1}(h)=P^{-1}(P(\varphi^{-1}(h)))=P^{-1}(h-R(\varphi^{-1}(h))),
$$
hence
$$
\varphi^{-1}(h)-P^{-1}(h)=P^{-1}(h-R(\varphi^{-1}(h)))-P^{-1}(h)= o(|h|^k),
$$
because $\varphi^{-1}$ and $P^{-1}$ are locally bi-Lipschitz, and $R(v)=o(|v|^k)$.
Now, $P^{-1}$ need not be a polynomial, but if we define $Q$ as the Taylor polynomial of order $k$ of $P^{-1}$ we then have
$$
P^{-1}(h)-Q(h)=o(|h|^k),
$$ 
and by summing the last two equations we obtain that
$$
\varphi^{-1}(h)-Q(h)=o(|h|^k),
$$
thus proving the first assertion of the Lemma. The second assertion is a consequence of the first one and of the fact that a $C^1$ diffeomorphism between open subsets of $\R^n$ maps $\mathcal{L}^{n}$-null sets onto $\mathcal{L}^{n}$-null sets.
\end{proof}

\begin{lemma}\label{inverses are OK2}
Let $k\geq 2$ be a positive integer and $\varphi:U\to V$ be a $C^{1}$ diffeomorphism between two open subsets of $\R^n$. Assume that $\varphi$  satisfies $$\limsup_{h\to 0}\frac{|\varphi(x+h)-P(h)|}{|h|^k}<\infty
$$ for a point $x\in U$ and a polynomial $P$ of degree up to $k-1$. Then the inverse $\varphi^{-1}$ satisfies
$$\limsup_{h\to 0}\frac{|\varphi^{-1}(y+h)-Q(h)|}{|h|^k}<\infty
$$ for $y=\varphi(x)$ and some polynomial $Q$ of order up to $k-1$.
\end{lemma}
\begin{proof}
Replace $o(|h|^k)$ with $O(|h|^k)$, and $P$ with the Taylor polynomial of order $k-1$ of $\varphi$ at $x$, in the proof of Lemma \ref{inverses are OK1}.
\end{proof}

\begin{lemma}\label{composition is OK1}
Assume that $f:\R^d\to\R^n$ has a Taylor expansion of order $k$ at $x$, and that $g:\R^n\to\R^m$ has a Taylor expansion of order $k$ at $y=f(x)$.
Then $g\circ f$ has a Taylor expansion of order $k$ at $x$.
\end{lemma}
\begin{proof}
We may assume that $x=y=0$, and write $f=Q+S$, $g=P+R$, where $P$, $Q$ are polynomials of degree less than or equal to $k$, $S(h)=o(|h|^k)$, and $R(h)=o(|h|^k)$. Then, using that $P$ and $Q$ are locally Lipschitz, it is easy to check that
\begin{eqnarray*}
& g(f(h)) & = P(f(h))+R(f(h))\\
& & =P(Q(h)+S(h))+R(Q(h)+S(h))\\
& & =P(Q(h))+\left(P(Q(h)+S(h))-P(Q(h))\right)+R(Q(h)+S(h))\\
& & =P(Q(h)) +o(|h|^k)+o(|h|^k),
\end{eqnarray*}
that is
$$
g(f(h))-P(Q(h))=o(|h|^k).
$$
Although the polynomial $P\circ Q$ is, in general, of order greater than $k$, we may define $T$ as the Taylor polynomial of order $k$ of $P\circ Q$ at $0$, so that
$$
P(Q(h))-T(h)=o(|h|^k),
$$
and by summing the last two equalities we get $g(f(h))-T(h)=o(|h|^k)$.
\end{proof}
               
\begin{lemma}\label{composition is OK2}
Assume that $f:\R^d\to\R^n$ satisfies
$$
\limsup_{h\to 0}\frac{|f(x+h)-Q(h)|}{|h|^k}<\infty$$ 
for some $x\in\R^d$ and a polynomial $Q$ of degree up to $k-1$, and that $g:\R^n\to\R^m$ satisfies
$$
\limsup_{h\to 0}\frac{|g(y+h)-P(h)|}{|h|^k}<\infty$$ 
for $y=f(x)$ and a polynomial $P$ of degree up to $k-1$. Then 
$$
\limsup_{h\to 0}\frac{|(g\circ f)(x+h)-T(h)|}{|h|^k}<\infty,$$
where $T$ is the Taylor polynomial of order $k-1$ of $P\circ Q$.
\end{lemma}
\begin{proof}
Replace $o(|h|^k)$ with $O(|h|^k)$, and $P$, $Q$, $T$ with polinomials of order $k-1$ in the proof of the preceding Lemma.
\end{proof}

\begin{lemma}\label{expansions in ae point imply expansions in ae in ae subspace}
Let $m,n,k, i$ be positive integers, with $i\leq n-1$. Assume that $f:\R^n\to\R^m$ has a Taylor expansion of order $k$ at $\mathcal{L}^{n}$-a.e. $x\in\R^n$. Then, for $\mathcal{L}^{i}$-a.e. $y\in\R^i$, the function $g_y:\R^{n-i}\to\R^m$ has a Taylor expansion of order $k$ at $\mathcal{L}^{n-i}$-a.e. $z\in\R^{n-i}$, where $g_{y}$ is defined by $g_{y}(z)=g(y,z)$.
\end{lemma}
\begin{proof}
Using the hypothesis and Fubini's theorem, for $\mathcal{L}^{i}$-a.e. $y\in\R^{i}$ we have that, for $\mathcal{L}^{n-i}$-a.e. $z\in \R^{n-i}$, the function $g$ has a Taylor expansion of order $k$ at $x=(y,z)$. This obviously implies that the function $g_{y}$ has a Taylor expansion of order $k$ at $z$, for such $y\in\R^{i}$, $z\in\R^{n-i}$.
\end{proof}

\medskip

\section{Proof of Theorem \ref{abstract theorem}}

Throughout this section we will assume that $f\in \mathcal{C}_{n,m}^{k}(U,\R^m)$, where $U\subseteq\R^n$ is open and $n>m\geq 1$ are positive integers, and $k$ will be defined by $k:=n-m+1$; note in particular that $k\geq 2$. We will also use $\mathcal{C}_{n,m}^{j}$ as an abbreviation for $\mathcal{C}_{n,m}^{j}(W,V)$ when the sets $W,V$ are understood.

Let $C=\{x\in U \, : \, \textrm{ rank}(Df(x))<m\}$ be the set of critical points of $f$, and set
$$
A_j=\{x\in C \, : \, D^{i}f(x)=0 \textrm{ for } 1\leq i\leq j\}, \, \, \, 1\leq j\leq n-m,
$$
and
$$
K=\{ x\in C: 1\leq \textrm{rank}(Df(x))\leq m-1\}.
$$
Then we have
\begin{equation}\label{decomposition Aj}
C=K\cup (A_1\setminus A_2)\cup (A_2\setminus A_3)\cup ...\cup (A_{n-m-1}\setminus A_{n-m})\cup A_{n-m}.
\end{equation}
The proof of Theorem \ref{abstract theorem} can be carried out in three steps following the general plan of \cite[Lemma 15.2]{AbrahamRobbin} and \cite[Theorem 5]{Figalli}. In the first step, by using conditions (MS4)-(MS7) and a standard argument going back to \cite{Sard} we show that it suffices to prove the theorem in the case $K=\emptyset$, that is, $C=\{x : Df(x)=0\}$. In the second step we prove that $\mathcal{L}^m(f(A_{n-m}))=0$. This is where our proof really differs from others; the key point is showing that for a function $\varphi:\R\to\R$ the set $\{t\in\R : (0,...,0, a_t)\in J^{k}\varphi(t) \textrm{ for some } a_t \neq 0\}$ is always of measure zero in $\R$. Finally, in the third step we see how, thanks to the Implicit Function Theorem and the Kneser-Glaeser Theorem, one can reduce the dimension from $n$ to $n-1$, which by an induction argument finishes the proof. This third step is almost identical to Step 3 in the proof of \cite[Theorem 5]{Figalli} or \cite[Lemma 15.2]{AbrahamRobbin}, but since the argument is short we will include its proof here for completeness (and for the benefit of those readers who might not yet be familiar with this scheme). 

\begin{claim}
{\bf First Step}. We may assume $K=\emptyset$.
\end{claim}
\begin{proof}
Observe that $f$ is at least of class $C^1$, because $n>m$ and $\mathcal{C}_{n,m}^{k}=\mathcal{C}_{n,m}^{n-m+1}\subset C^{n-m}$, by (MS1). Define
$$
K_i=\{x\in\R^n : \textrm{ rank}(D(f(x))=i\}, \, \, \, 1\leq i\leq m-1.
$$
Take $\bar{x}\in K_i$. We may assume that
$$
\textrm{det}\left(\frac{\partial (f_1, ..., f_i)}{\partial (x_1, ..., x_i)}\right)(\bar{x})\neq 0.
$$
Then, in some relatively compact neighborhood $V$ of $\bar{x}$, we can take as coordinates $(y_1, ..., y_n)=Y(x)=(f_1(x), ..., f_i(x), x_{i+1}, ..., x_{n})$. Therefore, defining $X=Y^{-1}$, $f$ is of the form
$$
f(X(y))=(y_1, ..., y_i, g(y_1,..., y_n)).
$$
By using properties (MS1), (MS4), (MS6) and (MS7) of the classes $\mathcal{C}_{j,\ell}^{i}$, it is easy to see that 
$X\in \mathcal{C}_{n,n}^{k}$. Then, by using  property (MS5) we get $f\circ X\in\mathcal{C}_{n,m}^{k}$, and by using again (MS5) and (MS1) it is immediately clear that $g\in\mathcal{C}_{n, m-i}^{k}(\widetilde{V}, \R^{m-i})$, where $\widetilde{V}=Y(V)$. Moreover, in these new coordinates we have
$$
D(f\circ X)(y)=
    \begin{pmatrix}
      I_{i} & 0 \\
      \ast & D(g_{|_{(y_1, ..., y_i)}}) \;
    \end{pmatrix},
$$
where
$g_{|_{(y_1, ..., y_i)}}(y_{i+1},...,y_n) :=g(y_1, ..., y_n)$, and $\widetilde{V}_{(y_1,...,y_i)}=\{(z_1, ..., z_{n-i})\in\R^{n-i} : (y_1,..., y_i, z_1,...,z_{n-i})\in \widetilde{V}\}$.

Now, by using property (MS6) and observing that $\textrm{rank}(D(f\circ X))=\textrm{rank}(Df)=i$ on $K_i$, we have that, for $\mathcal{L}^{i}$-a.e. $(y_1, ..., y_i)$,
$$
g_{|_{(y_1,..., y_i)}}\in \mathcal{C}_{n-i, m-i}^{k},
$$
and
$$
D(g_{|_{(y_1, ..., y_i)}})=0 \textrm{ on } Y(K_i)\cap \widetilde{V}_{(y_1,...,y_i)}.
$$ 
Therefore, if we prove that
$f(A_1\setminus A_2)$, ..., $f(A_{n-m-1}\setminus A_{n-m})$, $f(A_{n-m})$ are of measure zero (which we will indeed do in the second and third steps), by applying that part of the proof to the function $g_{|_{(y_1, ..., y_i)}}$, we will get
$$
\mathcal{L}^{m-i}(g_{|_{(y_1, ..., y_i)}}(Y(K_i)\cap \widetilde{V}_{(y_1,...,y_i)}))=0 \textrm{ for } \mathcal{L}^{i}\textrm{-a.e. } (y_1, ..., y_i),
$$ 
and by Fubini's theorem we will conclude that $0=\mathcal{L}^{m}(f\circ X(\widetilde{V}\cap Y(K_i)))$ $=\mathcal{L}^{m}(f(V\cap K_i))$, for every $i\in\{1, ..., m-1\}$, hence that $\mathcal{L}^{m}(f(K))=0$ as well.
\end{proof}

Let us now proceed with the {\bf Second Step} of the proof of Theorem \ref{abstract theorem} and show that $\mathcal{L}^m(f(A_{n-m}))=0$.
We can distinguish two subsets of $A_{n-m}$, namely,
$$
C_{n-m}=\{x\in A_{n-m} \, : \,  (0,...,0)\in J^{k}f(x)\},
$$
and $$B_{n-m}=A_{n-m}\setminus C_{n-m}.$$ By Lemma \ref{points with 0 jets are good} we already know that $\mathcal{L}^m(f(C_{n-m}))=0$. Therefore it is enough to see that $\mathcal{L}^m(f(B_{n-m}))=0$. In turn, because $B_{n-m}\subset A_{n-m}$, and thanks to property (MS2) of the class $\mathcal{C}_{n,m}^{k}$, this will be established once we prove the following.

\begin{lemma}\label{Bn-1 is null}
$\mathcal{L}^n(B_{n-m})=0$.
\end{lemma}
\begin{proof}
Property (MS3) tells us that $$\mathcal{L}^n\left(\{x\in\R^n : J^{k}f(x)=\emptyset\}\right)=0.$$
Therefore we may assume that
$$
B_{n-m}\subset\{x\in A_{n-m} : J^{k}f(x)\neq\emptyset\}.
$$
Since $C_{n-m}$ is disjoint with $B_{n-m}$ we then have that, for every $x\in B_{n-m}$,
$$
J^{k}f(x)=(0, ..., 0, P_{x}^k) \textrm{ for some } P_{x}^k\neq 0.
$$
Let $\{v_j\}_{j=1}^{\infty}$ and $\{w_{j}\}_{j=1}^{\infty}$ be dense sequences in the unit spheres $\mathbb{S}^{n-1}$ and $\mathbb{S}^{m-1}$ of $\mathbb{R}^n$ and $\mathbb{R}^{m}$, respectively. Define, for every $(\alpha, \beta)\in\N\times \N$,
$$
E_{\alpha \beta}^{+}=\{x\in B_{n-m} : \langle P_{x}^{k}(v_{\alpha}), w_{\beta}\rangle >0\},
$$ 
and
$$
E_{\alpha \beta}^{-}=\{x\in B_{n-m} : \langle P_{x}^{k}(v_{\alpha}), w_{\beta}\rangle <0\}.
$$ 
If $x\in B_{n-m}$ then $P_{x}^k(v)\neq 0$ for some $v\in\mathbb{S}^{n-1}$, and by continuity of $P_{x}^k$ and density of $\{v_j\}_{j=1}^{\infty}$, $\{w_j\}_{j=1}^{\infty}$ there exist $\alpha, \beta\in\N$ such that $\langle P_{x}^n(v_{\alpha}), w_{\beta}\rangle \neq 0$, and therefore $x\in E_{\alpha \beta}^{+}\cup E_{\alpha \beta}^{-}$. This shows that
$$
B_{n-m}=\bigcup_{\alpha, \beta\in\N}\left(E_{\alpha \beta}^{+}\cup E_{\alpha \beta}^{-}\right).
$$
Thus it is sufficient to see that $\mathcal{L}^n(E_{\alpha \beta}^{+})=0$ and $\mathcal{L}^n(E_{\alpha \beta}^{-})=0$ for each $(\alpha, \beta)\in\N\times\N$. Let us fix $\alpha, \beta\in\N$ and see that $\mathcal{L}^n(E_{\alpha \beta}^{+})=0$ (the proof that $\mathcal{L}^n(E_{\alpha \beta}^{-})=0$ is completely analogous). 

Let $[v_\alpha]^{\perp}$ stand for the orthogonal complement of the line spanned by the direction $v_\alpha$. For each $y\in [v_\alpha]^{\perp}$ let us define
\begin{eqnarray*}
& & R_{y}=\{y+tv_\alpha : t\in\R\}, \\
& & G_{y}=R_{y}\cap E_{\alpha \beta}^{+},
 \textrm{ and } \\
& & \widetilde{G_y}=\{t\in\R : y+tv_\alpha\in G_y\}.
\end{eqnarray*}
We will show that $\mathcal{L}^{1}(\widetilde{G_y})=0$. Once we have checked this, by applying Fubini's theorem we will immediately deduce that 
$\mathcal{L}^n(E_{\alpha \beta}^{+})=0$ (note that we may indeed apply Fubini's theorem because the sets $E_{\alpha \beta}^{\pm}$ are measurable; indeed, $A_{n-m}$ is closed, and we can write 
$$
C_{n-m}=\bigcap_{i=1}^{\infty}\bigcup_{j=1}^{\infty}\bigcap_{|h|\leq 1/j}\{x\in A_{n-m}: |f(x+h)-f(x)|\leq\frac{1}{i}|h|^k\},
$$
so that $C_{n-m}$ is  a countable intersection of $F_{\sigma}$ sets, hence $B_{n-m}$ is measurable,
and 
$$
E_{\alpha \beta}^{+}= B_{n-m}\cap
\left( \bigcup_{i=1}^{\infty}\bigcup_{j=1}^{\infty}\bigcap_{0\leq t\leq  1/j}
\{x : \langle f(x+tv_{\alpha}), w_{\beta}\rangle-\langle f(x), w_{\beta}\rangle \geq \frac{1}{i} t^k\}  \right),
$$
so that $E_{\alpha \beta}^{+}$, being the intersection of $B_{n-m}$ with an $F_{\sigma}$ set, is measurable as well. Similarly, $E_{\alpha \beta}^{-}$ is measurable).

So let us fix $y\in [v_\alpha]^{\perp}$, and consider the auxiliary function $\varphi:\R\to\R$ defined by 
$$
\varphi(t)=\langle f(y+tv_\alpha), w_{\beta}\rangle.
$$
For each $t\in \widetilde{G_y}$, recalling that $(0, ..., 0, P_{y+t v_{\alpha}})\in J^{k}f(y+t v_{\alpha})$, we have
$$
J^{k}\varphi(t)=\left( 0, ..., 0, \langle P_{y+tv_\alpha}^{k}(v_\alpha), w_{\beta}\rangle\right).
$$
In particular, for every $t\in \widetilde{G_y}$ we have $$J^k\varphi(t)=(0, ..., 0, a_t) \textrm{ for some } a_t>0.$$ 

Let us now distinguish two cases.

\noindent{\bf Case 1. Assume first that $k$ is even.}  In this case we will be able to show that in fact $\widetilde{G_y}$ is (at most) countable. In order to do so, let us define, for every $i,j, \ell \in\N$, the sets
$$
D_{\ell}=\{t\in \widetilde{G_y} : J^{k}\varphi(t)= (0,...,  0, a_t), \, a_t\geq\frac{1}{\ell}\}.
$$
and
$$
G_{i j \ell}=\{t\in \widetilde{G_y} \, : \, \varphi(t+h)-\varphi(t)-\frac{1}{\ell}h^k\geq -\frac{1}{j}h^k \textrm{ whenever } |h|\leq \frac{1}{i}\}.
$$
We obviously have $\widetilde{G_y}=\bigcup_{\ell=1}^{\infty}D_{\ell}$, hence we only have to check that $D_{\ell}$ is countable for each $\ell\in\N$.
Notice that
$$
D_{\ell}\subseteq \bigcap_{j=1}^{\infty}\left(\bigcup_{i=1}^{\infty}G_{ij\ell}\right)
$$
(indeed, for every $t\in D_{\ell}$ and for every $j\in\N$, since $\lim_{h\to 0}\frac{\varphi(t+h)-\varphi(t)-a_t h^k}{h^k}=0$ there exists $i\in\N$ so that $|\varphi(t+h)-\varphi(t)-a_t h^k|\leq h^k/j$
if $|h|\leq 1/i$. Hence, using that $a_t\geq 1/\ell$ because $t\in D_{\ell}$, we have
$$
\frac{1}{\ell}h^k\leq a_t h^k\leq \varphi(t+h)-\varphi(t)+\frac{1}{j}h^k \textrm{ for } |h|\leq 1/i,
$$
and in particular $t\in G_{ij\ell}$). 

Now, if $t,s\in G_{ij\ell}$ and $|s-t|\leq 1/i$ then we have
$$
\varphi(t+s-t)-\varphi(t)-\frac{1}{\ell}|t-s|^k\geq -\frac{1}{j}|t-s|^k
$$
and
$$
\varphi(s+t-s)-\varphi(s)-\frac{1}{\ell}|t-s|^k\geq -\frac{1}{j}|t-s|^k.
$$
By summing the last two inequalities we get
$$
-\frac{2}{\ell}|t-s|^k\geq -\frac{2}{j}|t-s|^k,
$$
and for $j>\ell$ this is possible only if $t=s$. Therefore, if $j>\ell$, the set $G_{ij\ell}$ contains only isolated points (as the distance between two of its points must be greater than $1/i$), and in particular  is countable. Then, if we set $j=2\ell$ for instance, it follows that
$$
D_{\ell}\subset \bigcup_{i=1}^{\infty}G_{i (2\ell) \ell}
$$ is countable.

\medskip

\noindent{\bf Case 2. Assume now that $k$ is odd.} Let $\{r_{\ell}\}_{\ell\in\N}$ be an enumeration of the positive rational numbers, and for each $i, j, \ell\in\N$ define functions
$$
g_{\ell,j}(h)=r_{\ell}h^k-\frac{2}{j}|h|^k,
$$
and sets
$$
D_{ij\ell}=\{t\in\widetilde{G_y} : \varphi(t+h)-\varphi(t)\geq g_{\ell,j}(h) \textrm{ if } |h|\leq \frac{1}{i}\}.
$$
For each $t\in\widetilde{G_y}$ we can find $\ell, j\in\N$ such that
$$
a_{t}\in [r_{\ell}-\frac{1}{j}, r_{\ell}+\frac{1}{j}], \, \textrm{ and } \, \frac{1}{j}\leq\frac{r_{\ell}}{4}.
$$
For these $j, \ell$, because $\lim_{h\to 0}\frac{\varphi(t+h)-\varphi(t)-a_{t}h^k}{|h|^k}=0$, there exists $i\in\N$ such that if $|h|\leq 1/i$ then
$$
\varphi(t+h)-\varphi(t)-a_{t}h^k\geq -\frac{1}{j}|h|^k.
$$
For $0\leq h\leq 1/i$ we then have
$$
\varphi(t+h)-\varphi(t)\geq a_{t}h^k-\frac{1}{j}|h|^k \geq (r_{\ell}-\frac{1}{j})h^k-\frac{j}{j}|h|^k=g_{\ell,j}(h),
$$
and for $-1/i\leq h\leq0$ we get
$$
\varphi(t+h)-\varphi(t)\geq a_{t}h^k-\frac{1}{j}|h|^k \geq (r_{\ell}+\frac{1}{j})h^k-\frac{1}{j}|h|^k=g_{\ell,j}(h).
$$
In either case we obtain $t\in D_{ij\ell}$. This shows that
$$
\widetilde{G_y}\subset\bigcup_{i,j,\ell\in\N, 4<j r_{\ell}}D_{i j \ell}.
$$
Hence it is enough to see that $\mathcal{L}^{1}(D_{ij\ell})=0$ for every $i,j,\ell\in\N$ with $4< j r_{\ell}$. To this end we can further assume without loss of generality that $D_{ij\ell}$ is bounded, and then
it will sufficient to show the following.
\begin{claim}
Let $I$ be a bounded interval in $\R$, let $k\in\N$ be odd, $c, \varepsilon$ be positive numbers with $2\varepsilon<c$, $\varphi:I\to \mathbb{R}$ be a function, and $D$ be a subset of $I$. Suppose that 
$$
\varphi(y)\geq \varphi(x)+c(y-x)^k-\varepsilon |y-x|^k
$$
for every $x,y\in D$. Then
$\mathcal{L}^{1}(D)=0$.
\end{claim}
\noindent In order to prove the Claim we use Lemma \ref{Cantor}. To check that $D$ has the property in the statement of Lemma \ref{Cantor}, take $x,y\in D$ and assume without loss of generality that $0=x<y$ and $\varphi(x)=0$. We have
$$
\varphi(y)\in \bigl( (c-\varepsilon )y^k,(c+\varepsilon )y^k\bigr) \subset \bigl( \frac{c}{2}y^k,\frac{3c}{2}y^k \bigr).
$$
If $z\in (x,y)\cap D$, we also have
$$
\varphi(z)\in \bigl( (c-\varepsilon )z^k,(c+\varepsilon )z^k\bigr) \subset \bigl( \frac{c}{2}z^k,\frac{3c}{2}z^k \bigr).
$$
On the other hand,
$$
\varphi(z)\geq \varphi(y)+c(z-y)^k-\varepsilon (y-z)^k\geq \frac{c}{2}y^k+(c+\varepsilon )(z-y)^k\geq
$$
$$
\geq \frac{c}{2}y^k+\frac{3c}{2}(z-y)^k,
$$
which implies
$$
\frac{c}{2}y^k+\frac{3c}{2}(z-y)^k\leq \varphi(z)\leq \frac{3c}{2}z^k,
$$
hence also
$$
y^k\leq 3((y-z)^k+z^k),
$$
which in turn implies
$$
z\not\in \bigl( \frac{y}{3},\frac{2y}{3}\bigr),
$$
because if $z\in ( \frac{y}{3},\frac{2y}{3})$ then 
$$
y^k\leq 3\bigl( (\frac{2}{3})^k y^k+(\frac{2}{3})^k y^k\bigr) =6(\frac{2}{3})^k y^k,
$$
and consequently 
$$
1\leq 6(\frac{2}{3})^k,
$$ 
which is impossible for $k\geq 5$; on the other hand, for $k=3$ one easily checks that $z\not\in (\frac{y}{3},\frac{2y}{3})$ by a straightforward calculation. So we may apply Lemma \ref{Cantor} to conclude the proof of the Claim. The proof of Lemma \ref{Bn-1 is null} is complete.
\end{proof}

\medskip 

Let us now finally make the {\bf Third Step} of the proof of Theorem \ref{abstract theorem}. Note that the above argument already proves the Theorem for $n=m+1$ (or equivalently $k=2$). This allows us to start an induction argument on the dimension $n$. Assuming that the result is true for functions $f$ in the class $\mathcal{C}_{n-1,m}^{n-m}$ with $K=\emptyset$, we have to check that it is also true for functions $f$ in the class $\mathcal{C}_{n,m}^{n-m+1}$ with $K=\emptyset$. By the second step we know that $\mathcal{L}^{m}(f(A_{n-m}))=0$. Therefore, bearing in mind equation $(\ref{decomposition Aj})$, we only have to show the following.
\begin{claim}
$\mathcal{L}^m(f(A_{s-1}\setminus A_{s}))=0$ for $2\leq s\leq k-1$.
\end{claim}
\begin{proof}
Fix a point $x\in A_{s-1}\setminus A_{s}$. It is sufficient to see that there exists an open neighborhood $V$ of $x$ such that $\mathcal{L}^{m}(f((A_{s-1}\setminus A_{s})\cap V))=0$. By the standing hypotheses, $f$ is of class $C^{k-1}$, and $f$ is $(s-1)$-flat at $x$, but some partial derivative of order $s$ of $f$ is not zero. We may assume for instance that
$$\frac{\partial w}{\partial x_n}(x)\neq 0,$$ where $$w(x)=\partial^{\alpha}f(x)=0$$ and $\alpha$ is a multi-index with $|\alpha|=s-1$. Since $w$ is of class $C^{k-s}$, by the Implicit Function Theorem there exists an open neighborhood $V$ of $x$ such that $V\cap\{w=0\}$ is an $(n-1)$-dimensional graph of class $C^{k-s}$. In particular there exist an open subset $W$ in $\R^{n-1}$ and a function $g:W\to \R^n$ of class $C^{n-s}$ such that $V\cap A_{s-1}\subset g(W)$.

Now let us consider $A^{*}:=\{z\in W : g(z)\in A_{s-1}\}$. By Theorem \ref{Kneser Glaeser}, there exists $F:W\to\R^m$ of class $C^{k-1}$ such that
$F=f\circ g$ on $A^{*}$, and $DF(x)=0$ on $A^{*}$. In particular we have
$f(A_{s-1}\cap V)\subset F(C_{F}\cap W)$, where $C_{F}$ stands for the critical set of $F$. But, since $F$ is defined on an open subset of $\R^{n-1}$ and is of class $C^{k-1}$, and by property (MS1) we have $C^{k-1}\subset \mathcal{C}_{n-1,m}^{k-1}=\mathcal{C}_{n-1,m}^{n-m}$, we may use the induction hypothesis to conclude that $\mathcal{L}^{m}(F(C_{F}\cap W))=0$, and therefore $\mathcal{L}^m(f(A_{s-1}\cap V))=0$ as well. The proof is complete.

Finally, notice that conditions (MS4)-(MS7) are used only in the First Step of the above proof. Since in the special case $m=1$ we always have $K=\emptyset$, the First Step of the proof may be omitted. Therefore it is not necessary to define the classes $C^{j}_{n,1}(U,V)$ for $j\neq n$, and the conditions (MS1) for $j=n$, (MS2) and (MS3) alone are sufficient to ensure that $\mathcal{L}^{1}(f(C_f))=0$ for every $f\in \mathcal{C}_{n,1}^{n}(U,V)$.
\end{proof}

\begin{remark}\label{simplicity}
{\em It is worth noting that if we replace condition $(MS3)$ with
\begin{itemize}
\item[{(MS3')}] if $n\geq m+1$ then every $f\in\mathcal{C}_{n,m}^{n-m+1}(U,V)$ is $n-m+1$ times differentiable at almost every $x\in\R^n$,
\end{itemize} 
then there is no need to consider Case 2 of the proof of Lemma \ref{Bn-1 is null}, and we obtain a simpler proof of an easier variant of Theorem \ref{abstract theorem} which still is powerful enough to imply Theorems \ref{De Pascale} and \ref{Stepanov meets Morse-Sard}. The changes one has to make in order to obtain this simpler proof are as follows: observe that if $f$ is $k$ times differentiable at a point $x$ then $D^{k-2}f$ is twice differentiable at $x$ and in particular has a Taylor expansion of order $2$ at $x$; then apply the proof of Case 1 of Lemma \ref{Bn-1 is null} to the function 
$$
\varphi(t)=\langle D^{k-2}f(y+tv_\alpha)(v_{\alpha}^{k-2}), w_{\beta}\rangle = D^{k-2}\langle f, w_{\beta}\rangle (y+t v_\alpha)(v_{\alpha}^{k-2})
$$
(where now $k$ is not necessarily even). We have  $$J^2\varphi(t)=(0, a_t) \textrm{ for some } a_t>0 \textrm{ for every }t\in \widetilde{G_y},$$ hence a particular case of that proof allows us to deduce that $\widetilde{G_y}$ is countable.}
\end{remark}

\medskip

\section{Proof of Theorem \ref{general theorem}}

The very particular case $n=m$ in Theorem \ref{general theorem} is well known, see \cite[Proposition 3.6]{Howard} for instance. Now, in order to prove Theorem \ref{general theorem} in the case $n>m$, we define, for each $U$ open in $\R^i$ and $V$ open in $\R^j$ with $i\geq j$, the class  $\mathcal{C}_{i,j}^{s}(U,V)$ as the set of all functions $f:U\to V$ such that:
\begin{enumerate}
\item $f\in C^{s-1}(U, V)$;
\item $\limsup_{h\to 0}\frac{|f(x+h)-f(x)-Df(x)(h) - ... - \frac{1}{(s-1)!} D^{s-1}f(x)(h^{s-1})|}{|h|^s}<\infty$ for every $x\in\R^n$; 
\item $f$ has a Taylor expansion of order $s$ at almost every $x\in U$.
\end{enumerate}
Note that, by results of Liu and Tai \cite{LiuTai}, every function $f$ satisfying $(2)$ also satisfies $(3)$.
By using Lemmas \ref{n null sets are mapped onto m null sets}, \ref{inverses are OK1}, \ref{inverses are OK2}, \ref{composition is OK1}, \ref{composition is OK2}, and \ref{expansions in ae point imply expansions in ae in ae subspace}, and bearing in mind that diffeomorphisms map (Lebesgue) null sets onto null sets, it is easy to verify that the classes $\mathcal{C}_{i,j}^{s}$ satisfy the properties (MS1)--(MS7) of Theorem \ref{abstract theorem}. Then Theorem \ref{general theorem} for $n>m$ follows at once.

\medskip

\section{Examples and relations to other work}

Let us begin by briefly explaining how De Pascale's Theorem \ref{De Pascale}  also follows, with little effort, from Theorem \ref{abstract theorem}.  In the case $n>m$, we may define, for each $U$ open in $\R^i$ and $V$ open in $\R^j$, with $i\geq j$, $s\geq 1$, the classes $\mathcal{C}_{i,j}^{s}(U,V):=W^{s,p}_{\textrm{loc}}(U,V)$. By using standard results and techniques of Sobolev space theory (Change of Variables, Slicing Theorem, etc) it is not difficult to check that these classes satisfy properties (MS4)-(MS7). Property (MS1) is part of Morrey's inequality, and property (MS2) is the easy part of Step 2 in the proof of \cite[Theorem 5]{Figalli} (to check this, one may combine Taylor's theorem, Morrey's inequality, and Young's inequality to see that
$$
|f(y)-f(x)|^{m}\leq C\int_{ B(x,r)}(1+|D^{k}f(z)|^p)dz,
$$
for $x\in A$ and $|y-x|\leq r\leq 1$, and then conclude by a covering argument, see \cite{Figalli} for details). Finally, property (MS3)
can be easily checked as follows: consider the map $\R^n \ni x\mapsto g(x)=D^{k-1}f(x)\in\R^{M}$, where $M$ is the dimension of the space $\mathcal{L}_{s}^{k-1}(\R^n, \R^m)$ of $(k-1)$-linear symmetric maps from $\R^n$ to $\R^m$ and we identify $\mathcal{L}_{s}^{k-1}(\R^n, \R^m)$ with $\R^{M}$. Using that $f\in W^{k,p}_{\textrm{loc}}(\R^n, \R^m)$, an easy calculation shows that the coordinate functions $g_j$ of $g=(g_1,..., g_M)$ have first order weak derivatives which are in $L^{p}$. That is, $g_{j}\in W^{1,p}_{\textrm{loc}}(\R^n)$ for every $j=1, ..., M$. By \cite[Theorem 6.2.1]{EvansGariepy}, $g_j$ is then differentiable almost everywhere in $\R^n$, for each $j=1, ..., M$. It follows that $g$ is differentiable almost everywhere in $\R^m$, which means that $f$ is $k$ times differentiable at a.e. $x\in\R^n$, and in particular $f$ has a Taylor expansion at a.e. $x\in\R^n$. Thus we may apply Theorem \ref{abstract theorem} to deduce that every $f\in W^{k,p}_{\textrm{loc}}(\R^n, \R^m)$ has the Morse-Sard property if $k=n-m+1\geq 2$ and $p>n$.

As for the case $n=m$, this is an easy consequence of the coarea formula for Sobolev mappings \cite{MalySwansonZiemer} (or an immediate consequence of the case $n=m$ of Theorem \ref{Stepanov meets Morse-Sard} above and the mentioned fact that functions of $W^{1,p}_{\textrm{loc}}(\R^n, \R^m)$ are differentiable almost everywhere when $p>n$).

\medskip

Let us now see why Theorems \ref{general theorem} and \ref{Stepanov meets Morse-Sard}, and Corollary \ref{main theorem} are not weaker than the recent Bourgain-Korobkov-Kristensen generalizations \cite{BourKoKris2} of the Morse-Sard theorem in the case of real-valued functions for the spaces $W^{n,1}(\R^n, \R)$ and $BV_{n, \textrm{loc}}(\R^n)$.
Since $W^{1,1}_{\textrm{loc}}(\R^n)\subset BV_{\textrm{loc}}(\R^n)$, the results of \cite{BourKoKris2} are stronger than all the previous generalizations of the Morse-Sard theorem for Sobolev spaces in the case $m=1$, and are also stronger than \cite{PavZaj} and \cite[Theorem 8]{Rifford} (even in the case $n=2$, because the results of \cite{BourKoKris2} do not require that the function be Lipschitz, and because every locally semiconcave function on $\R^2$ belongs to $BV_{2, \textrm{loc}}(\R^2)$). Thus, in order to make our point, it will be enough to exhibit examples of functions $f\in\mathcal{C}^1\mathcal{P}^2(\R^2, \R)$ such that $f\notin BV_{2, \textrm{loc}}(\R^2)$.

For a simple, explicitly defined example, let us consider $f:\R^2\to\R$,
    $$
    f(x,y)=
  \begin{cases}
    x^4\sin\left(\frac{1}{x^2}\right) & \text{ if } x\neq 0, \\
    0 & \text{ if } x=0.
  \end{cases}
    $$
We have $\frac{\partial f}{\partial y}=0$ everywhere, and
    $$
    \frac{\partial f}{\partial x}(x,y)=
  \begin{cases}
    4x^3\sin\left(\frac{1}{x^2}\right)-2x\cos\left(\frac{1}{x^2}\right) & \text{ if } x\neq 0, \\
    0 & \text{ if } x=0,
  \end{cases}
    $$
so that $f\in C^1(\R^2)$. On the other hand, it is clear that $f$ is $C^{\infty}$ on $\{ (x,y) : x\neq 0\}$, and for every $(x_0, y_0)$ with $x_0=0$ we have that $(0,0)\in J^{2}f(0,y_0)$, because
$$
\lim_{(x,y)\to (0,y_0)}\frac{x^4\sin\left(\frac{1}{x^2}\right)}{x^2+y^2}=0.
$$
Hence $f\in \mathcal{C}^1\mathcal{P}^2(\R^2, \R)$. However, $g:=\partial f/\partial x\notin BV(\R^2)$. Indeed, defining
$g_{y}(x)=\frac{\partial f}{\partial x}(x,y)$,
it is easy to see that
$$
V_{0}^{1}g_{y}=\infty
$$ for each $y\in\R$ (here we use the notation from \cite[5.11]{EvansGariepy}), and because $g_y$ is continuous this implies
$$
\textrm{ess}V_{0}^{1}g_{y}=\infty
$$ for every $y$, hence 
$$
\int_{-1}^{1}\textrm{ess}V_{0}^{1}g_{y}=\infty,
$$
and by \cite[Theorem 5.11.2]{EvansGariepy} this implies that $g\notin BV_{\textrm{loc}}(\R^2)$, hence $f\notin BV_{2, \textrm{loc}}(\R^2)$ either. 

If the reader wishes to look at more complex examples with sets of critical points of positive measure where the functions are not locally $BV_2$ (which prevents the application of all of the previously known results in order to obtain the Morse-Sard property), he or she might want to consider the following.

\begin{ex}
{\em Let $C\subset [0,1]$ be a Cantor-like set of positive measure.
Construct a continuous function $g:[0,1]\to \R$ as follows. Set $g(x)=0$ for every $x\in C$ and, for
each of the $2^{n-1}$ intervals $I_n^j$ of length $l_n$ that are removed from an interval $I^{k}_{n-1}$ at step $n$ in the construction of $C$,
consider a subinterval $J_n^j$ of length $\frac{l_n}{3}$ centered at the same point as $I_n^j$. Define $g$ on $I_n^j$ as a differentiable function which is not of bounded variation and such that $0\leq g(x)\leq l_n^{\frac{3}{2}}$ and $g(x)=0$ for every
$x\in I_n^j\setminus J_n^j$.
The function $F:(0,1)^2\to \R$ defined by $F(x,y)=f(x)+f(y)$ with $f(x)=\int_0^xg(t)dt$ satisfies $C\times C\subset C_F$ and is twice differentiable at every point, but it does not have a $BV$ derivative. However, $F$ satisfies the hypotheses of Theorem \ref{Stepanov meets Morse-Sard}, and consequently has the Morse Sard property. }
\end{ex}

On the other hand, it is clear that the results of \cite{BourKoKris2} and, in the case $n=2$, \cite{PavZaj} and \cite[Theorem 8]{Rifford} are not weaker than Theorem \ref{general theorem} either. For instance, it is easy to produce examples of delta-convex, and of locally semiconcave, functions $f:\R^2\to\R$ which are not of class $C^1$. 

However, if $f:\R^2\to\R$ is locally semiconcave, the parts of the proofs of Theorems \ref{abstract theorem} and \ref{general theorem} that are relevant to this situation can easily be adapted to show that $\mathcal{L}^{1}\left(f(\{x\in\R^2 : Df(x) \textrm{ exists and is zero } \})\right)$ $=0$, thus recovering \cite[Theorem 8]{Rifford} and the main result of \cite{Landis}, see also \cite{PavZaj}. In order to do so one only has to note that: 1) all d.c. convex functions are locally semiconcave; 2) all locally semiconcave functions have second order Taylor expansions at almost every point (thanks to Alexandroff's theorem); and 3) if $f$ is locally semiconcave then the proof of Lemma \ref{n null sets are mapped onto m null sets} can be easily adapted to show that $\mathcal{L}^{1}(f(N))=0$ for every subset $N$ of $\{x\in\R^2 : Df(x) \textrm{ exists and is zero} \}$ with $\mathcal{L}^{2}(N)=0$. As a matter of fact, one can also adapt the proofs of Theorems \ref{abstract theorem} and \ref{general theorem} to show that if $f:\R^n\to\R$ is of class $C^{n-2}$ and the directional derivatives of order $n-2$ of $f$ are locally semiconcave functions, with constants of local semiconcavity that are independent of the directions, then $f$ has the Morse-Sard property. We will not spell out the details because in view of \cite[Theorem 6.3.3]{EvansGariepy} this result is also an immediate consequence of the Bourgain-Korobkov-Kristensen theorem for $BV_{n}$ functions in \cite{BourKoKris2}.

Let us finish this section with some remarks about the statement of Theorem \ref{general theorem}. A natural question is whether condition $(2)$ of Theorem \ref{general theorem} could be replaced with a weaker condition in which the limit would hold for a.e. $x$ instead of all $x$. The answer is negative. Indeed, consider Whitney's example from \cite{Whitney}, where an arc $C$ in $\R^2$ and a function $f:\R^2\to\R$ are constructed in such a way that $f$ is critical on $C$ (meaning that $C\subseteq C_f$) and $f$ is not constant on $C$; in particular $f(C_f)$ contains an open interval, and $f$ does not have the Morse-Sard property. See also \cite{Norton2, Hajlasz2} and the references therein for more information about Whitney-type examples. Although not explicitly stated in Whitney's paper, these $f$ and $C$ satisfy two important additional properties:
\begin{enumerate}
\item[{(a)}] $f$ is of class $C^{\infty}$ on $\R^2\setminus C$, and
\item[{(b)}] $\mathcal{L}^{2}(C)=0$.
\end{enumerate}
Property (b) follows easily from the definition of $C$, while (a) is a consequence of the facts that $f$ is constructed by applying the Whitney Extension Theorem to a function defined on $C$, and that Whitney's theorem provides us with extensions which are always of class $C^\infty$ outside the closed set on which the functions to be extended are initially defined.
Then, since $f$ has derivatives of all orders which are locally Lipschitz on the open set $\R^2\setminus C$, it is clear that $\limsup_{h\to 0}\frac{|f(x+h)-f(x)-Df(x)(h)|}{|h|^2}<\infty$ for all $x\in \R^2\setminus C$, and in particular for a.e. $x\in\R^2$. On the other hand $f$ clearly satisfies conditions $(1)$ and $(3)$ of Theorem \ref{general theorem}. However, $\mathcal{L}^{1}(f(C_f))>0$.

Thus, refining the question about condition $(2)$ of Theorem \ref{general theorem}, one could ask: how small must a set $N\subset\R^n$ be in order that Theorem \ref{general theorem} still holds true if we replace condition $(2)$ with a new condition in which the limit holds for every $x\in\R^n\setminus N$? This question is of course much more difficult to answer. By combining Theorem \ref{abstract theorem} for $m=1$ with \cite[Theorem 2]{Norton} and the results of \cite{LiuTai}, we can nevertheless obtain a partial answer as follows.

\begin{theorem}\label{general theorem variant 1}
Let $n\geq 2$, $0<\alpha<1$, and let $f:\R^n\to\R$ be such that
\begin{enumerate}
\item $f\in C^{n-1, \alpha}(\R^n, \R)$;
\item $\limsup_{h\to 0}\frac{|f(x+h)-f(x)-Df(x)(h) - ... - \frac{1}{(n-1)!} D^{n-1}f(x)(h^{n-1})|}{|h|^n}<\infty$ for every $x\in\R^n\setminus N$, for some $N$ with $\mathcal{H}^{n-1+\alpha}(N)=0$.
\end{enumerate}
Then $\mathcal{L}^{1}\left( f(C_f) \right)=0$.
\end{theorem}
\noindent Here $C^{n-1, \alpha}$ denotes the subset of $C^{n-1}$ defined by all functions whose derivatives of order $n-1$ satisfy H\"older-continuity conditions of order $\alpha$ on compact sets.
\begin{proof} Let us define, for each $U$ open in $\R^n$ the class  $\mathcal{C}_{n,1}^{n}(U,\R)$ as the set of all functions $f:U\to \R$ such that:
\begin{enumerate}
\item $f\in C^{n-1, \alpha}(U, \R)$;
\item $\limsup_{h\to 0}\frac{|f(x+h)-f(x)-Df(x)(h) - ... - \frac{1}{(n-1)!} D^{n-1}f(x)(h^{n-1})|}{|h|^n}<\infty$ for every $x\in U\setminus N$, for some $N\subset\R^n$ with $\mathcal{H}^{n-1+\alpha}(N)=0$; 
\item $f$ has a Taylor expansion of order $n$ at almost every $x\in U$.
\end{enumerate}
By the results of \cite{LiuTai} we have that every function $f$ satisfying $(2)$ also satisfies $(3)$.
It is clear that the classes $\mathcal{C}_{n,1}^{n}$ satisfy properties (MS1) and (MS3) of Theorem \ref{abstract theorem} for $m=1, j=n$.
As for property (MS2), we know by \cite[Theorem 2(ii)]{Norton} that every $f\in C^{j+\alpha}(\R^n, \R)$ maps $\mathcal{H}^{j+\alpha}$-null critical sets onto $\mathcal{L}^{1}$-null sets. Therefore, by Lemma \ref{n null sets are mapped onto m null sets}, (MS2) will be satisfied as long as we ask that $|f(x+h)-f(x)-Df(x)(h) - ... - \frac{1}{(n-1)!} D^{n-1}f(x)(h^{n-1})|=O(|h|^n)$ for every $x$ outside an $\mathcal{H}^{n-1+\alpha}$-null set.
\end{proof}
\begin{remark}
{\em The reason why we cannot similarly apply Theorem \ref{abstract theorem} in the case $m\geq 2$ is that we cannot check condition (MS6) due to the lack of a Fubini theorem for Hausdorff measures.}
\end{remark}

By using the same argument, and taking into account that $C^{n-1}\subset C^{n-2, \alpha}$, one may also prove the following. 
\begin{theorem}\label{general theorem variant 2}
Let $n\geq 2$, $0<\alpha<1$, and let $f:\R^n\to\R$ be such that
\begin{enumerate}
\item $f\in C^{n-1}(\R^n, \R)$;
\item $\limsup_{h\to 0}\frac{|f(x+h)-f(x)-Df(x)(h) - ... - \frac{1}{(n-1)!} D^{n-1}f(x)(h^{n-1})|}{|h|^n}<\infty$ for every $x\in\R^n\setminus N$, for some $N$ with $\mathcal{H}^{n-2+\alpha}(N)=0$.
\end{enumerate}
Then $\mathcal{L}^{1}\left( f(C_f) \right)=0$.
\end{theorem}

The reader is invited to consider other classes of functions (e.g. $C^{k+\beta+}$ in \cite{Norton}) and use Theorem \ref{abstract theorem} to formulate other variants of Theorem \ref{general theorem}.

\medskip

\section{Proofs of Theorems \ref{abstract D-S theorem}, \ref{general D-S theorem}, and \ref{HajlaszZimmerman}}

The proof of Theorem \ref{abstract D-S theorem} follows the same plan as that of Theorem \ref{abstract theorem}.
Given a function $f\in\mathcal{C}_{n,m}^{k}$, we define $C=\{x\in U \, : \, \textrm{ rank}(Df(x))<m\}$, and set
$$
A_j=\{x\in C \, : \, D^{i}f(x)=0 \textrm{ for } 1\leq i\leq j\}, \, \, \, 1\leq j\leq k-1,
$$
and
$$
K=\{ x\in C: 1\leq \textrm{rank}(Df(x))\leq m-1\},
$$
so that
\begin{equation}\label{decomposition Aj bis}
C=K\cup (A_1\setminus A_2)\cup (A_2\setminus A_3)\cup ...\cup (A_{k-2}\setminus A_{k-1})\cup A_{k-1}.
\end{equation}
Steps 1 and 3 of the proof of Theorem \ref{abstract theorem} can be rewritten, with obvious changes and no difficulty, to see that one can always assume $K=\emptyset$ and that, once one has checked that $\mathcal{H}^{\ell}(A_{k-1}\cap f^{-1}(y))=0$ for a.e. $y$, one can also use the Implicit Function Theorem and the Kneser-Glaeser Theorem in order to reduce the dimension from $n$ to $n-1$ and apply the induction hypothesis.

So, the only really different point of the proof is Step 2. Let us see that $\mathcal{H}^{\ell}(A_{k-1}\cap f^{-1}(y))=0$ for a.e. $y\in\R^m$.
We can distinguish two subsets of $A_{k-1}$, namely,
$$
C_{k-1}=\{x\in A_{k-1} \, : \,  (0,...,0)\in J^{k}f(x)\},
$$
and $$B_{k-1}=A_{k-1}\setminus C_{k-1}.$$ 
\begin{claim}\label{points with O jets are OK in DS theorem}
$\mathcal{H}^{\ell}(C_{k-1}\cap f^{-1}(y))=0$ for almost every $y\in\R^m$.
\end{claim}
\begin{proof}
It suffices to prove the Claim in the case $\ell\geq 1$, as the case $\ell=0$ is just Theorem \ref{abstract theorem}. We may assume that $C_{k-1}\subset B:=[-R,R]^{n}$ for some fixed $R\in\N$. Let $\varepsilon >0$. For every $i\in \N$ we define $D_i$ as in the proof of Lemma \ref{points with 0 jets are good}, so we have $C_{k-1}\subset \bigcup _j D_j $ and $D_{i}\subset D_{i+1}$. For each $i$, we decompose $B$ as the union of a family of cubes $\{Q_{ij}\}_{j=1}^{N(i)}$ of diameter  $\sqrt{n}/i$ with pairwise disjoint interiors. 
If $x,y\in D_i\cap Q_{ij}$ equation \eqref{diameters well bounded} in the proof of Lemma \ref{points with 0 jets are good} shows that 
\begin{equation}\label{estimation of Hausdorf measure 1}
\mathcal{H}^m(f(D_i\cap Q_{ij}))\leq C(m,n)\textrm{diam}(f(D_i\cap Q_{ij}))^{m} \leq \varepsilon^{m} C(m,n)\textrm{diam}(Q_{ij})^{km},
\end{equation}
where $C(m,n)$ denotes a constant only depending on $m, n$. For all $s\geq i$, the sets $\{D_{s}\cap Q_{sj}\cap f^{-1}(y)\}_{j}$ form a covering of $D_{i}\cap f^{-1}(y)$ by sets of diameter less than or equal to $\sqrt{n}/s$, and $\textrm{diam}(D_{s}\cap Q_{sj}\cap f^{-1}(y))\leq \textrm{diam}(Q_{sj})\chi_{f(D_{s}\cap Q_{sj})}(y)$. So we have, for all $\alpha\geq s\geq i$, 
\begin{eqnarray*}
& &\mathcal{H}^{\ell}_{\sqrt{n}/s}(D_{i}\cap f^{-1}(y))\leq
C(\ell) \sum_{j}\textrm{diam}(D_{\alpha}\cap Q_{\alpha j}\cap f^{-1}(y))^{\ell}\\
& &\leq C(\ell) \sum_{j}\textrm{diam}(Q_{\alpha j})^{\ell}\chi_{f(D_{\alpha}\cap Q_{\alpha j})}(y),
\end{eqnarray*}
hence
$$
\mathcal{H}^{\ell}_{\sqrt{n}/s}(D_{i}\cap f^{-1}(y))\leq 
\liminf_{\alpha\to\infty} C(\ell) \sum_{j}\textrm{diam}(Q_{\alpha j})^{\ell}\chi_{f(D_{\alpha}\cap Q_{\alpha j})}(y),
$$ 
and by letting $s\to\infty$ we get
$$
\mathcal{H}^{\ell}(D_{i}\cap f^{-1}(y))\leq 
\liminf_{\alpha\to\infty} C(\ell) \sum_{j}\textrm{diam}(Q_{\alpha j})^{\ell}\chi_{f(D_{\alpha}\cap Q_{\alpha j})}(y),
$$
and consequently
\begin{eqnarray*}
& &\mathcal{H}^{\ell}(C_{k-1}\cap f^{-1}(y))\leq \sup_{i}\mathcal{H}^{\ell}(D_{i}\cap f^{-1}(y))\\
& &\leq
\liminf_{\alpha\to\infty} C(\ell) \sum_{j}\textrm{diam}(Q_{\alpha j})^{\ell}\chi_{f(D_{\alpha}\cap Q_{\alpha j})}(y).
\end{eqnarray*}
We now use a variation of an idea in \cite[Claim 3.1]{HajlaszZimmerman}: by taking the upper integral with respect to $d\mathcal{H}^{m}(y)$ on both sides of the above inequality (see \cite{EvansGariepy} for the definition of the upper integral), using Fatou's Lemma, plugging \eqref{estimation of Hausdorf measure 1}, and observing that $\ell+km=n+(m-1)(k-1)\geq n$, we obtain
\begin{eqnarray*}
& &\int^{*}_{\R^m}\mathcal{H}^{\ell}(C_{k-1}\cap f^{-1}(y))d\mathcal{H}^{m}(y) \\
& &\leq C(m,n) \liminf_{i\to\infty}\sum_{j}\textrm{diam}(Q_{ij})^{\ell}\mathcal{H}^{m}(f(D_{i}\cap Q_{ij}))\\
& &\leq C(\ell, m,n)\liminf_{i\to\infty} \sum_{j}\textrm{diam}(Q_{ij})^{\ell} \varepsilon^{m}\textrm{diam}(Q_{i,j})^{km}\\
& &\leq  
C(\ell, m,n)\liminf_{i\to\infty} \sum_{j} \varepsilon^{m}\textrm{diam}(Q_{i,j})^{n}\leq C(\ell, m)(2R)^{n}\varepsilon^{m}.
\end{eqnarray*}
By sending $\varepsilon$ to $0$ we thus have
$$
\int^{*}_{\R^m}\mathcal{H}^{\ell}(C_{k-1}\cap f^{-1}(y))d\mathcal{H}^{m}(y)=0.
$$
This means that there exists a sequence $\{\varphi_j\}$ of simple $\mathcal{H}^{m}$-measurable functions such that  $\varphi_j(y)\geq \mathcal{H}^{\ell}(C_{k-1}\cap f^{-1}(y))$ and $\int_{\R^m}\varphi_j (y) d\mathcal{H}^{m}(y)\to 0$. So $\varphi_{j}\to 0$ in $L^{1}(\R^m, \mathcal{H}^{m})=L^{1}(\R^m)$, hence there is a subsequence $\varphi_{j_i}$ such that $\lim_{i\to\infty}\varphi_{j_i}(y)=0$ for $\mathcal{L}^{m}$-a.e. $y$, which implies $\mathcal{H}^{\ell}(C_{k-1}\cap f^{-1}(y))=0$ for $\mathcal{L}^{m}$-a.e. $y\in\R^m$.
\end{proof}

Now, notice that the proof of Lemma \ref{Bn-1 is null} also shows (just by replacing $n-m$ with $k-1$) that $\mathcal{L}^{n}(B_{k-1})=0$ in the current setting. Therefore, by using condition (DS2) we also get 
$\mathcal{H}^{\ell}(B_{k-1}\cap f^{-1}(y))=0$ for a.e. $y\in\R^m$, which, together with Claim \ref{points with O jets are OK in DS theorem} yields $\mathcal{H}^{\ell}(A_{k-1}\cap f^{-1}(y))=0$ for a.e. $y\in\R^m$. The proof of Theorem \ref{abstract D-S theorem} is complete.

\medskip

Next, in order to deduce Theorem \ref{general D-S theorem} from Theorem \ref{abstract D-S theorem}, we define, for each $U$ open in $\R^i$ and $V$ open in $\R^j$ with $i\geq j$, and $1\leq s\leq n$, the class  $\mathcal{C}_{i,j}^{s}(U,V)$ as the set of all functions $f:U\to V$ such that:
\begin{enumerate}
\item $f\in C^{s-1}(\R^i, \R^j)$;
\item $\limsup_{h\to 0}\frac{|f(x+h)-f(x)-Df(x)(h) - ... - \frac{1}{(s-1)!} D^{s-1}f(x)(h^{s-1})|}{|h|^s}<\infty$ for every $x\in\R^n$; 
\item $f$ has a Taylor expansion of order $s$ at almost every $x\in\R^n$.
\end{enumerate}
Again, by the results of \cite{LiuTai} every function that satisfies $(2)$ also satisfies $(3)$.
By adding the following Lemma to Lemmas \ref{n null sets are mapped onto m null sets}, \ref{inverses are OK1}, \ref{inverses are OK2}, \ref{composition is OK1}, \ref{composition is OK2}, \ref{expansions in ae point imply expansions in ae in ae subspace}, and using the fact that diffeomorphisms map null sets onto null sets it is easy to check that the classes $\mathcal{C}_{i,j}^{s}$ satisfy properties (DS1)--(DS7) of Theorem \ref{abstract D-S theorem}; thus Theorem \ref{general D-S theorem} follows from Theorem \ref{abstract D-S theorem} (note that the case $n=m$ corresponds to $k=1$ and $\ell=0$, a situation which is already covered by Theorem \ref{general theorem}). 
\begin{lemma}\label{what happens with null sets in the D-S result}
Let $f:\R^n \to\R^m$ be a function, $1\leq k\leq n-m+1$, and define $\ell=n-m-k+1$. Then, for every subset $N$ of  $\{x : \limsup_{h\to 0}\frac{|f(x+h)-f(x)|}{|h|^k}<\infty\}$ such that $\mathcal{L}^n(N)=0$, we have that $\mathcal{H}^{\ell}(N\cap f^{-1}(y))=0$ for almost every $y\in\R^m$.
\end{lemma}
\begin{proof}
In the case $\ell=0$, which corresponds to $k=n-m+1$, this Lemma tells us the same thing as Lemma \ref{n null sets are mapped onto m null sets}. Therefore we may assume $\ell\geq 1$.
For each $j\in\N$ let us define $A_j$ as in the proof of Lemma \ref{n null sets are mapped onto m null sets}. It is enough to see that $\mathcal{H}^{\ell}(A_{j}\cap N\cap f^{-1}(y))=0$ for every $j\in\N$. So fix $j\in\N$ and $\varepsilon>0$, and for each $\alpha\in\N$ choose a sequence of cubes $\{Q_{\alpha\beta}\}_{\beta}$ such that $\textrm{diam}(Q_{\alpha\beta})\leq 1/\alpha$, $N\subset \bigcup_{\beta=1}^{\infty}Q_{\alpha\beta}$, and $\sum_{\beta=1}^{\infty}\textrm{diam}(Q_{\alpha\beta})^{n}\leq\varepsilon$. Define now $D_i$ as in the proof of Lemma \ref{n null sets are mapped onto m null sets}, so that $D_{i}\subset D_{i+1}$ and $A_{j}\cap N=\bigcup_{i=1}^{\infty}D_i$. If $x,y\in D_{i}\cap Q_{i\beta}$ we have, as in the proof of Lemma \ref{n null sets are mapped onto m null sets}, that
$$
|f(x)-f(y)|\leq (j+1)|x-y|^{k},
$$
which implies that
$$
\mathcal{H}^{m}(f(D_{i}\cap Q_{i\beta}))\leq C(m)(j+1)^{m}\textrm{diam}(Q_{i\beta})^{km}.
$$
Now, as in the proof of Lemma \ref{points with O jets are OK in DS theorem}, for every $\alpha\geq i$ the sets $\{D_{\alpha}\cap Q_{\alpha\beta}\cap f^{-1}(y)\}_{\beta}$ form a covering of $A_{j}\cap N\cap D_{i}\cap f^{-1}(y)$ by sets of diameter less than or equal to $\sqrt{n}/\alpha$, and $\textrm{diam}(D_{\alpha}\cap Q_{\alpha\beta}\cap f^{-1}(y))\leq \textrm{diam}(Q_{\alpha\beta})\chi_{f(D_{\alpha}\cap Q_{\alpha\beta})}(y)$ so,
by the same argument as in the proof of Lemma \ref{points with O jets are OK in DS theorem}, we have
$$
\mathcal{H}^{\ell}(A_{j}\cap N\cap f^{-1}(y))\leq 
\liminf_{\alpha\to\infty} C(\ell) \sum_{\beta}\textrm{diam}(Q_{\alpha \beta})^{\ell}\chi_{f(D_{\alpha}\cap Q_{\alpha \beta})}(y).
$$
Then, similarly to the proof of Lemma \ref{points with O jets are OK in DS theorem}, we deduce that
\begin{eqnarray*}
& &\int^{*}_{\R^m}\mathcal{H}^{\ell}(A_{j}\cap N\cap f^{-1}(y))d\mathcal{H}^{m}(y) \\
& &\leq C(m,n) \liminf_{\alpha\to\infty}\sum_{\beta}\textrm{diam}(Q_{\alpha\beta})^{\ell}\mathcal{H}^{m}(f(D_{\alpha}\cap Q_{\alpha\beta}))\\
& &\leq C(\ell, m,n)\liminf_{\alpha\to\infty} \sum_{\beta}\textrm{diam}(Q_{\alpha\beta})^{\ell} (j+1)^{m}\textrm{diam}(Q_{\alpha\beta})^{km}\\
& &\leq  
C(\ell, m,n)\liminf_{\alpha\to\infty} \sum_{j} (j+1)^{m}\textrm{diam}(Q_{\alpha\beta})^{n}\leq C(\ell, m,n) (j+1)^{m} \varepsilon^{m}.
\end{eqnarray*}
By letting $\varepsilon$ go to $0$ we thus have
$
\int^{*}_{\R^m}\mathcal{H}^{\ell}(A_{j}\cap N\cap f^{-1}(y))d\mathcal{H}^{m}(y)=0,
$
and consequently that $\mathcal{H}^{\ell}(A_{j}\cap N\cap f^{-1}(y))=0$ for a.e. $y\in\R^m$.
\end{proof}

Finally, let us say a few words about how one can deduce Theorem \ref{HajlaszZimmerman} from Theorem \ref{abstract D-S theorem}.  One can define, for each $U$ open in $\R^i$ and $V$ open in $\R^j$, with $i\geq j$, $s\geq 1$, the classes $\mathcal{C}_{i,j}^{s}(U,V):=W^{s,p}_{\textrm{loc}}(U,V)$. By using standard results and techniques of Sobolev space theory it is not difficult to check that these classes satisfy properties (DS4)-(DS7). Property (DS1) follows from Morrey's inequality, and property (DS3) has already been checked above (in the proof of Theorem \ref{De Pascale} that we included at the beginning of Section 5). The only delicate point is thus checking property (DS2). This is implicitly shown, by combining Morrey's and H\"older's inequalities with a clever use of the upper integral, in \cite[Claim 3.1, Step 1]{HajlaszZimmerman}, to where we refer the reader for details.

\medskip

\noindent {\bf Acknowledgement:} This work was finished while D. Azagra was at Equipe d'Analyse Fonctionelle de l'Institut de Math\'ematiques de Jussieu, Paris. D. Azagra wishes to thank l'Equipe, and very specially Gilles Godefroy and Yves Raynaud for their kind hospitality.

\end{document}